\documentclass[draftclsnofoot,onecolumn,12pt]{IEEEtran}

\usepackage{amsmath,amssymb,enumerate,subfigure,multirow,hhline,color}
\usepackage{xcolor}
\usepackage{hyperref}
\usepackage{graphicx}
\usepackage{graphics}
\usepackage[sort]{cite}
\usepackage{amsthm}
\usepackage{algorithm}
\usepackage{algpseudocode}

\usepackage{epsfig,psfrag}
\usepackage{tabularx}
\usepackage{tabulary}

\usepackage[sort]{cite}

\usepackage{hyperref}
\hypersetup{breaklinks=true,colorlinks=true,linkcolor=blue,citecolor=blue}

\usepackage[noabbrev,capitalise,nameinlink]{cleveref}

\DeclareMathOperator{\arginf}{arg\,inf}

\newtheorem{example}{Example}

\newtheorem{problem}{Problem}
\newtheorem{lemma}{Lemma}

\newtheorem{proposition}{Proposition}

\allowdisplaybreaks

\title{\LARGE \bf Supplemental Material For ``Primal-Dual Q-Learning Framework for LQR Design''}

\author{Donghwan Lee and Jianghai Hu
\thanks{D. Lee is with Coordinated Science Laboratory (CSL), University of Illinois at Urbana-Champaign, IL 61801, USA {\tt\small donghwan@illinois.edu}.}
\thanks{J. Hu is with the Department of Electrical and Computer Engineering,
Purdue University, West Lafayette, IN 47906, USA {\tt\small
jianghai@purdue.edu}.}%
}

\begin{document}

\maketitle \thispagestyle{empty} \pagestyle{empty}
\begin{abstract}
Recently, reinforcement learning (RL) is receiving more and more attentions due to its successful demonstrations outperforming human performance in certain challenging tasks. In our recent paper `primal-dual Q-learning framework for LQR design,' we study a new optimization formulation of the linear quadratic regulator (LQR) problem via the Lagrangian duality theories in order to lay theoretical foundations of potentially effective RL algorithms. The new optimization problem includes the Q-function parameters so that it can be directly used to develop Q-learning algorithms, known to be one of the most popular RL algorithms. In the paper, we prove relations between saddle-points of the Lagrangian function and the optimal solutions of the Bellman equation. As an application, we propose a model-free primal-dual Q-learning algorithm to solve the LQR problem and demonstrate its validity through examples. It is meaningful to consider additional potential applications of the proposed analysis. Various SDP formulations of Problem 5 or Problem 2 of the paper can be derived, and they can be used to develop new analysis and control design approaches. For example, an SDP-based optimal control design with energy and input constraints can be derived. Another direction is algorithms for structured controller designs. These approaches are included in this supplemental material.
\end{abstract}

\section{Notation}

The adopted notation is as follows: ${\mathbb N}$
and ${\mathbb N}_+$: sets of nonnegative and positive integers,
respectively; ${\mathbb R}$: set of real numbers; ${\mathbb R}_+$:
set of nonnegative real numbers; ${\mathbb R}_{++}$: set of
positive real numbers; ${\mathbb R}^n $: $n$-dimensional Euclidean
space; ${\mathbb R}^{n \times m}$: set of all $n \times m$ real
matrices; $A^T$: transpose of matrix $A$; $A \succ 0$ ($A \prec
0$, $A \succeq 0$, and $A \preceq 0$, respectively): symmetric
positive definite (negative definite, positive semi-definite, and
negative semi-definite, respectively) matrix $A$; $I_n $: $n
\times n$ identity matrix; ${\mathbb S}^n $: symmetric $n \times
n$ matrices; ${\mathbb S}_+^n $: cone of symmetric $n \times n$
positive semi-definite matrices; ${\mathbb S}_ {++} ^n $:
symmetric $n \times n$ positive definite matrices; ${\bf Tr}(A)$:
trace of matrix $A$; $\rho(\cdot)$: spectral radius.

\section{Projected gradient descent method for structured optimal control design}
In this section, we study a projected gradient descent algorithm
to approximately solve the structured control design
problems~\cite{geromel1998static,lin2011augmented} (e.g., the
output feedback, decentralized, and distributed control designs).
The main ideas originate
from~\cite{lin2011augmented,Martensson2012}. For any fixed $F \in
{\mathbb R}^{n\times m}$ and initial state $z\in {\mathbb R}^n$,
consider the discounted cost
\begin{align}
&J_\alpha(F,z):=  {\sum\limits_{k = 0}^\infty{
\alpha^k\begin{bmatrix}
   {x(k;F,z)}  \\
   {Fx(k;F,z)}  \\
\end{bmatrix}^T \Lambda \begin{bmatrix}
   {x(k;F,z)}  \\
   {Fx(k;F,z)}  \\
\end{bmatrix}}}.\label{eq:J-alpha}
\end{align}
As in the previous sections, the cost $J_\alpha(F,z)$ can be
expressed as $J_\alpha(F,z) = {\bf Tr}(\Lambda S)$, where $S$
satisfies
\begin{align*}
&\alpha A_{F}SA_{F}^T+
\begin{bmatrix}
   I\\
   F\\
\end{bmatrix} zz^T \begin{bmatrix}
   I  \\
   F\\
\end{bmatrix}^T = S.
\end{align*}
Define the structured identity~\cite{lin2011augmented}
\begin{align*}
&[I_{\cal K}]_{ij}= \begin{cases}
 0\quad {\rm if}\,\,[F]_{ij}= 0\,\,\,{\rm is\,\,required} \\
 1\quad {\rm otherwise} \\
 \end{cases}
\end{align*}
where $[F]_{ij}$ indicates its element in $i$-th row and $j$-th
column. In addition, define the subspace
\begin{align*}
&{\cal K}: = \{ F \in {\mathbb R}^{n \times m} :F \circ I_{\cal K}
= F\},
\end{align*}
where $\circ$ denotes the entry-wise multiplication of matrices
(Hadamard product). The structured optimal control design problem
can be stated as follows.
\begin{problem}\label{problem:structured-LQR}
Solve
\begin{align*}
&\mathop{\inf}_{S \in {\mathbb S}^{n + m} ,\,F
\in {\mathbb R}^{m \times n} } \,\,{\bf Tr}(\Lambda S)\\
&{\rm subject\,\,to}\quad F \in {\cal K},\\
&\alpha A_F SA_F^T + \begin{bmatrix}
 I_n \\
 F \\
\end{bmatrix} zz^T \begin{bmatrix}
 I_n \\
 F \\
\end{bmatrix}^T  = S.
\end{align*}
\end{problem}
Now, we will compute the gradient of $J_\alpha(F,z)$. We follow
the main ideas
from~\cite{lin2011augmented,rautert1997computational,Martensson2012}.
For any matrix $X$, let $dX$ denote an infinitesimal change of the
variable $X$. For a matrix valued function $f: {\mathbb
R}^{n\times m} \to {\mathbb R}$, we define the differential $df$
as the part of $f(X+dX) - f(X)$ that is linear in $dX$. From the
proof of~[Lemma~4, main document], we can easily derive the
following result.
\begin{lemma}
Let $F \in {\cal F}$ be given. If $S \in {\mathbb S}_{+}^{n+m}$
satisfies $S = \alpha A_F SA_F^T  +
\begin{bmatrix}
   I  \\
   F  \\
\end{bmatrix} zz^T \begin{bmatrix}
   I  \\
   F  \\
\end{bmatrix}^T$, then, $\alpha \begin{bmatrix}
   A & B  \\
\end{bmatrix} S \begin{bmatrix}
   A & B  \\
\end{bmatrix}^T+zz^T=S_{11}$.
\end{lemma}
Based on the lemma, we can calculate the gradient of the cost
in~\eqref{eq:J-alpha}.
\begin{proposition}\label{proposition:gradient1}
We have
\begin{align*}
&\nabla_F J_\alpha(F,z) =2 P_{12}^T S_{11} + 2 P_{22} S_{12},
\end{align*}
where $S$ and $P$ are solutions to
\begin{align}
&S = \alpha A_{F} SA_{F}^T + \begin{bmatrix}
   I  \\
   F \\
\end{bmatrix} zz^T \begin{bmatrix}
   I  \\
   F \\
\end{bmatrix}^T.\label{eq:eq6}
\end{align}
for $S$ and $\alpha A_{F}^T PA_{F}-P+\Lambda=0$ for $P$,
respectively.
\end{proposition}
\begin{proof}
Consider the cost function $J_\alpha(F,z) = {\rm Tr}(\Lambda S)$,
where $S$ solves~\eqref{eq:eq6}. It is importance to notice that
$S$ is a function of $F$. Its differential with respect to $F$ is
$dJ_\alpha(F,z) = {\bf Tr}(\Lambda dS)$, where $dS =\alpha A_F
dSA_F^T  + N + N^T$ and
\begin{align*}
&N := \begin{bmatrix}
   0 & 0  \\
   dFS_{11}& dFS_{11} F^T\\
\end{bmatrix}.
\end{align*}
Since $dS$ satisfies the Lyapunov equation $dS =\alpha A_F dSA_F^T
+ N + N^T$, it can be rewritten by
\begin{align*}
&dS = \sum_{k=0}^\infty {\alpha^k (A_F)^k (N+N^T )(A_F^T)^k } =
2\sum_{k=0}^\infty {\alpha^k (A_F)^k N^T(A_F^T )^k }.
\end{align*}
Plug the above equation into $dJ_\alpha(F,z) = {\bf Tr}(\Lambda
dS)$ to have
\begin{align*}
&dJ_\alpha(F,z) = {\bf Tr}(\Lambda dS) = 2{\bf Tr}\left(\Lambda
\sum_{k=0}^\infty \alpha^k (A_F)^k N^T(A_F^T)^k\right)\\
&=2{\bf Tr}\left(\sum_{k=0}^\infty{\alpha^k (A_F^T)^k\Lambda
(A_F)^k} N^T\right) = 2{\bf Tr}(PN^T),
\end{align*}
Noting that
\begin{align*}
&PN^T  =\begin{bmatrix}
   0 & P_{11} S_{11} dF^T+ P_{12} FS_{11} dF^T \\
   0 & P_{12}^T S_{11} dF^T+P_{22} FS_{11} dF^T \\
\end{bmatrix},
\end{align*}
we have $dJ_\alpha(F,z)={\bf Tr}(dF^T(2 P_{12}^T S_{11}+2P_{22}
FS_{11}))$. Plugging $F=S_{12}S_{11}^{-1}$ in~[Lemma~4, main
document] into the last equation and using $df(F)={\bf Tr}(dF^T
X)\Leftrightarrow \nabla_F
f(F)=X$~\cite[pp.~840]{rautert1997computational} for any $X$ leads
to the desired result.
\end{proof}

To compute the gradient in~\cref{proposition:gradient1}, we need
to solve~\eqref{eq:eq6} for $S$ and $\alpha A_{F}^T
PA_{F}-P+\Lambda=0$ for $P$. If the model is known, then both $S$
and $P$ can be approximated from simulations. To this end, define
the adjoint system
\begin{align}
&\xi (k + 1) = A_F^T \xi (k),\quad \xi (0) = \xi\in {\mathbb
R}^{n+m},\label{eq:adjoint-system}
\end{align}
We denote $\xi(k;F,\xi)$ by the solution
of~\eqref{eq:adjoint-system} starting from $\xi(0)=\xi$ and let
$\{ \xi _i \}_{i=1}^r$ be a set of initial states such that
$\Lambda=\sum_{i=1}^r (\xi_i \xi_i^T)$, where $\Lambda :=
\begin{bmatrix}
Q & 0\\
0 & R\\ \end{bmatrix}$. Denote by $\Pi_{\cal K}$ the projection
onto ${\cal K}$. An approximate gradient descent algorithm to find
a suboptimal solution to~\cref{problem:structured-LQR} is given
in~\cref{algo:gradient-descent-iteration}, where
$(\gamma_t)_{t=0}^\infty$ is a step-size sequence. For any $F_t\in
{\mathbb R}^{n\times m}$, the projection $\Pi_{\cal K}$ can be
performed easily by $\Pi_{\cal K}(F_t)=F_t\circ I_{\cal K}$. We
note that similar and effective gradient descent algorithms have
been studied
previously~\cite{lin2011augmented,rautert1997computational,Martensson2012}.
The proposed~\cref{algo:gradient-descent-iteration} has a
different form, and can be extended to a model-free method
presented in the next section. The convergence
of~\cref{algo:gradient-descent-iteration} to a stationary point
can be proved based on the standard projected gradient descent
algorithm~\cite[Section~2.3]{bertsekas1999nonlinear}. The
constant, diminishing, or Armijo-Goldstein step size rules can be
applied to guarantee the convergence.
\begin{algorithm}[h]
\caption{Projected Gradient Descent Algorithm for Structured
Optimal Control Design}
\begin{algorithmic}[1]
\State Initialize $F_0$ and set $t=0$.

\Repeat

\State For a sufficiently large integer $M>0$, compute
\begin{align*}
&\tilde S(F_t) :=\sum_{k=0}^M{ \left( \alpha^k
\begin{bmatrix}
   x(k;F_t,z)\\
   F_t x(k;F_t,z)\\
\end{bmatrix}\begin{bmatrix}
   x(k;F_t,z)\\
   F_t x(k;F_t,z)\\
\end{bmatrix}^T \right)}.
\end{align*}

\State Compute
\begin{align*}
&d_t:= 2 \tilde S_{11}(F_t) \tilde P_{12}(F_t)+ 2 \tilde
S_{12}(F_t) \tilde P_{22}(F_t).
\end{align*}

\State $F_{t+1}=\Pi_{\cal K}(F_t-\gamma_t d_t)$.

\State $t \leftarrow t+1$

\Until{a certain stopping criterion is satisfied. }
\end{algorithmic}
\label{algo:gradient-descent-iteration}
\end{algorithm}
\begin{example}
Consider~[Example~1, main document] with identical assumptions
except for the fact that 1) the system model $(A,B)$ is known, and
2) only the indoor air temperature $x_1(k)$ ($^\circ C$) and the
reference temperature $x_4(k)$ ($^\circ C$) can be measured.
Therefore, we want to design an output feedback control policy
with the output $y(k) = Cx(k)$, $C:= \begin{bmatrix}
   1 & 0 & 0 & 0  \\
   0 & 0 & 0 & 0  \\
   0 & 0 & 0 & 0  \\
   0 & 0 & 0 & 1  \\
\end{bmatrix}$.~\cref{algo:gradient-descent-iteration} was applied
with $\alpha = 0.9$, $z:=\begin{bmatrix}  1 & -1 & 5 & 2 \\
\end{bmatrix}^T$, and $\gamma_t=0.001$. After $10000$ iterations,
we obtained $F_t=\begin{bmatrix}
 -1.8359 & 0 & 0 & 1.8120 \\
\end{bmatrix}$, and the simulation results are depicted
in~\cref{fig:ex-fig3} and~\cref{fig:ex-fig4}, where
\cref{fig:ex-fig3} includes the evolution of the cost
in~\cref{algo:gradient-descent-iteration} and~\cref{fig:ex-fig4}
illustrates the state trajectory with the design control policy
for several different initial states, demonstrating the validity
of the algorithm.
\begin{figure}[t]
\centering\includegraphics[width=8cm,height=6cm]{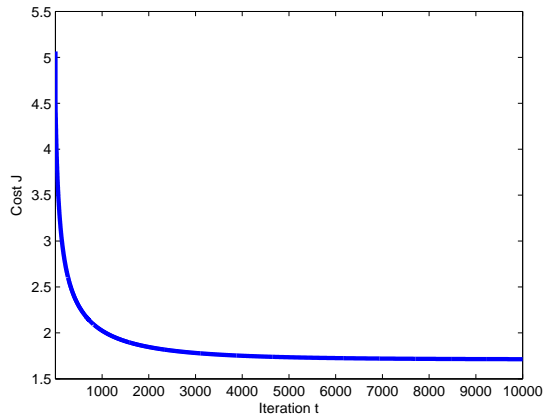}
\caption{Evolution of $J_\alpha(F_t,z)$.}\label{fig:ex-fig3}
\end{figure}
\begin{figure}[t]
\centering\includegraphics[width=7cm,height=4cm]{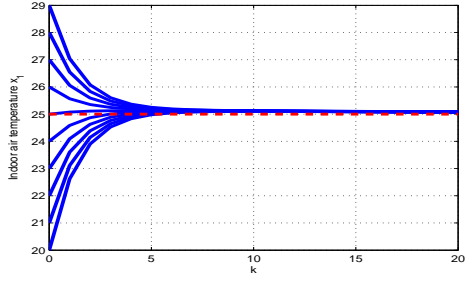}
\caption{Trajectories of $x_1(k)$ (indoor air temperature, blue
lines) under the designed structured LQR control
policy.}\label{fig:ex-fig4}
\end{figure}
\end{example}

\section{Model-free projected gradient descent method}
In this section, we propose a model-free version
of~\cref{algo:gradient-descent-iteration}. In~[Algorithm~2, main
document], one can observe that once $\tilde S(F_t)$ is obtained,
then $\tilde P(F_t)$ can be computed by solving the linear matrix
equation~[(25), main document] if $\tilde S(F_t)\succ 0$.
Similarly to~[Section~IV, main document], consider the
augmented state vector $v(k):=\begin{bmatrix} x(k)\\
u(k)\\ \end{bmatrix}$ and assume that we know the initial state
$v(0) = v_0 \in {\mathbb R}^{n+m}$. Denote by $v(k;F,v_0)$ the
state trajectory of the augmented system~[(3), main document] at
time $k$ starting from the initial augmented $\begin{bmatrix}
  x(0)\\
  u(0)\\
\end{bmatrix} = v_0$. As in~[Section~IV, main document], $u(0)$ can be freely chosen, and the control policy $u(k)=Fx(k)$ is valid from
$k=1$. One can choose $v_i\in {\mathbb R}^{n+m},i\in \{ 1,2,\ldots
,r\}$, such that $\sum_{i = 1}^r {v_i v_i^T} = \Gamma \succ 0$,
where $\Gamma \in {\mathbb S}^{n+m}$. Define
\begin{align*}
&J_\alpha(F,\Gamma ):= \sum_{i=1}^r {\sum_{k=0}^\infty {\alpha^k
v(k;F,v_i)^T \Lambda v(k;F,v_i )}}.
\end{align*}
The cost $J_\alpha(F,\Gamma)$ can be expressed as
$J_\alpha(F,\Gamma)= {\bf Tr}(\Lambda S)$, where $S$ satisfies
\begin{align*}
&\alpha A_{F}SA_{F}^T+ \Gamma = S.
\end{align*}
The structured optimal control design problem can be stated as
follows.
\begin{problem}\label{problem:structured-LQR2}
Solve
\begin{align*}
&\mathop{\inf}_{S \in {\mathbb S}^{n + m} ,\,F
\in {\mathbb R}^{m \times n} } \,\,{\bf Tr}(\Lambda S)\\
&{\rm subject\,\,to}\quad F \in {\cal K},\quad \alpha A_F SA_F^T +
\Gamma = S.
\end{align*}
\end{problem}
We will compute the gradient of $J_\alpha(F,\Gamma)$ following the
main ideas
of~\cite{lin2011augmented,rautert1997computational,Martensson2012}.
The following lemma can be easily proved.
\begin{lemma}
Let $F \in {\cal F}$ and $\Gamma \succ 0$ be given. If $S \in
{\mathbb S}_{+}^{n+m}$ satisfies $S=\alpha A_F SA_F^T +\Gamma$,
then $\alpha \begin{bmatrix}
   A & B\\
\end{bmatrix} S\begin{bmatrix}
   A & B\\
\end{bmatrix}^T +\Gamma_{11}=S_{11}$, where $\Gamma_{11}\in {\mathbb S}^{n\times
n}_{++}$ is the first $n$-by-$n$ block diagonal of $\Gamma\in
{\mathbb S}^{(n+m)\times (n+m)}_{++}$.
\end{lemma}
\begin{proposition}\label{proposition:model-free-gradient}
We have
\begin{align*}
&\nabla_F J(F,\Gamma)=2P_{12}^T
(S_{11}-\Gamma_{11})+2P_{22}F(S_{11}-\Gamma_{11}),
\end{align*}
where $S$ and $P$ are solutions to
\begin{align}
&S = \alpha A_{F} SA_{F}^T + \Gamma.\label{eq:eq7}
\end{align}
for $S$ and
\begin{align}
&\alpha A_{F}^T PA_{F}-P+\Lambda=0\label{eq:eq8}
\end{align}
for $P$, respectively.
\end{proposition}
\begin{proof}
Consider the cost function $J_\alpha(F,\Gamma) = {\rm Tr}(\Lambda
S)$, where $S$ solves~\eqref{eq:eq7}. Its differential with
respect to $F$ is $dJ_\alpha(F,\Gamma)={\bf Tr}(\Lambda dS)$,
where $dS=\alpha A_F dS A_F^T+N+N^T$,
\begin{align*}
&N= \begin{bmatrix}
   0 & 0 \\
   dFH & dFHF^T\\
\end{bmatrix},
\end{align*}
and $H :=S_{11}-\Gamma_{11}$. Since $dS$ satisfies the Lyapunov
equation $dS =\alpha A_F dSA_F^T + N + N^T$, it can be rewritten
by
\begin{align*}
&dS = \sum_{k=0}^\infty {\alpha^k (A_F)^k (N+N^T )(A_F^T)^k } =
2\sum_{k=0}^\infty {\alpha^k (A_F)^k N^T(A_F^T )^k }.
\end{align*}
Plug the above equation into $dJ_\alpha(F,z) = {\bf Tr}(\Lambda
dS)$ to have
\begin{align*}
&dJ_\alpha(F,z) = {\bf Tr}(\Lambda dS) = 2{\bf Tr}\left(\Lambda
\sum_{k=0}^\infty \alpha^k (A_F)^k N^T(A_F^T)^k \right)\\
&=2{\bf Tr}\left(\sum_{k=0}^\infty{\alpha^k (A_F^T)^k\Lambda
(A_F)^k} N^T\right) = 2{\bf Tr}(PN^T),
\end{align*}
Noting that
\begin{align*}
&PN^T = \begin{bmatrix}
   0& P_{11} HdF^T+P_{12} FHdF^T \\
   0& P_{12}^T HdF^T+P_{22} FHdF^T \\
\end{bmatrix},
\end{align*}
we have $dJ_\alpha  (F,\Gamma ) = {\bf{Tr}}(dF^T (2P_{12}^T H +
2P_{22} FH))$. Using $df(F) = {\bf Tr}(dF^T X) \Leftrightarrow
\nabla_F f(F)=X$~\cite[pp.~840]{rautert1997computational} for any
$X$ leads to the desired result.
\end{proof}
As in~[Section~IV, main document], $S$ solving~\eqref{eq:eq7} can
be obtained from state trajectories, $P$ solving~\eqref{eq:eq8}
can be obtained from the linear matrix equation~[(25), main
document], and the gradient
in~\cref{proposition:model-free-gradient} can be calculated. A
model-free algorithm based on these procedures is introduced
in~\cref{algo:model-free-gradient-descent-iteration}.
\begin{algorithm}[h]
\caption{Model-free Projected Gradient Descent Algorithm for
Structured Optimal Control Design}
\begin{algorithmic}[1]
\State Initialize $F_0$ and set $t=0$.

\Repeat

\State Compute
\begin{align*}
&\tilde S(F_t ): = \sum\limits_{i = 1}^r {\sum\limits_{k = 0}^M
{\alpha ^k v(k;F_t ,v_i )v(k;F_t ,v_i )^T } }.
\end{align*}

\State Compute $P(F_t)$ by solving for $P$
\begin{align*}
&\alpha W(F_t)^T PW(F_t ) + \tilde S(F_t )(\Lambda  - P)\tilde
S(F_t ) = 0,
\end{align*}
where
\begin{align*}
&W(F_t)=\sum_{i=1}^r {\sum_{k=0}^M {\alpha^k v(k+1;F_t,v_i
)v(k;F_t ,v_i)^T}}.
\end{align*}

\State Compute
\begin{align*}
&d_t =2P_{12}(F_t)^T (\tilde
S_{11}(F_t)-\Gamma_{11})\\
&+2P_{22}(F_t)F_t(\tilde S_{11}(F_t)-\Gamma_{11}),
\end{align*}

\State $F_{t+1}=\Pi_{\cal K}(F_t-\gamma_t d_t)$.

\State $t \leftarrow t+1$

\Until{a certain stopping criterion is satisfied. }
\end{algorithmic}
\label{algo:model-free-gradient-descent-iteration}
\end{algorithm}

\section{Semidefinite programming algorithms}
In this section, we study several semidefinite programming problem
(SDP) formulations of the LQR design problem with constraints. We
first introduce a modified version of the extended Schur
complement lemma in~\cite[Theorem 1]{de1999new}.
\begin{lemma}\label{lemma:extended-Schur-lemma}
The following conditions are equivalent:
\begin{enumerate}
\item There exists a symmetric matrix $P\succ 0$ such that $A^TPA-
P' \preceq 0$ hold.

\item There exist a matrix $G$ and a symmetric matrix $P\succ 0$
such that
\begin{align*}
&\begin{bmatrix}
   P' & A^T G^T \\
   GA & G+G^T-P \\
\end{bmatrix} \succeq 0.
\end{align*}
\end{enumerate}
\end{lemma}
\begin{proof}
1)$\Rightarrow$ 2): The proof is a modification of the proof
of~\cite[Theorem 1]{de1999new}. If the condition 1) holds, then by
the Schur complement~[Lemma~3, main document], we have
$\begin{bmatrix}
   P' & A^T \\
   A & P^{-1} \\
\end{bmatrix} \succeq 0 \Leftrightarrow \begin{bmatrix}
   P' & A^T P\\
   PA & P\\
\end{bmatrix} \succeq 0$. By setting $G = G^T = P$, the condition 2) is satisfied.

2)$\Rightarrow$ 1): Assume that the condition~2) holds. Pre- and
post-multiplying it by $\begin{bmatrix} I & -A^T \\ \end{bmatrix}$
and its transpose yield $A^TPA- P' \preceq 0$. Since $P$ is
already positive definite, the desired result follows.
\end{proof}
As a first step, we will introduce an SDP relaxation
of~[Problem~3, main document]. By~[Proposition~3, main document],
the optimization problem in~[Problem~3, main document] is
equivalent to the following problem.
\begin{problem}\label{problem:inequality-relaxation0}
Solve
\begin{align*}
&J_p:=\mathop {\inf}\limits_{S \in {\mathbb S}^{n + m} ,\,F
\in {\mathbb R}^{m \times n} } \,\,{\bf Tr}(\Lambda S)\\
&{\rm subject\,\,to}\quad S \succeq 0,\\
&A_F SA_F^T + \begin{bmatrix}
 I_n \\
 F \\
\end{bmatrix} Z \begin{bmatrix}
 I_n \\
 F \\
\end{bmatrix}^T = S.
\end{align*}
\end{problem}
By replacing the equality constraint into an inequality
constraint, we obtain the following problem.
\begin{problem}\label{problem:inequality-relaxation}
Solve
\begin{align*}
&J_{p,1}:=\mathop {\inf}\limits_{S \in {\mathbb S}^{n + m} ,\,F
\in {\mathbb R}^{m \times n} } \,\,{\bf Tr}(\Lambda S)\\
&{\rm subject\,\,to}\quad S \succeq 0,\\
&A_F SA_F^T + \begin{bmatrix}
 I_n \\
 F \\
\end{bmatrix} Z \begin{bmatrix}
 I_n \\
 F \\
\end{bmatrix}^T  \preceq S.
\end{align*}
\end{problem}
We can prove that~\cref{problem:inequality-relaxation} is
equivalent to~[Problem~3, main document].
\begin{proposition}\label{proposition:equivalence}
The optimal value of~\cref{problem:inequality-relaxation} is
equivalent to that of~[Problem~3, main document], i.e.,
$J_{p}=J_{p,1}$.
\end{proposition}
\begin{proof}
Since~\cref{problem:inequality-relaxation} has a larger feasible
set than~\cref{problem:inequality-relaxation0}, we have
$J_{p,1}\leq J_{p}$. To prove the reversed inequality, let
$(S_{p,1},F_{p,1})$ be an optimal solution
to~\cref{problem:inequality-relaxation} and $(S_{p},F_{p})$ be an
optimal solution to~\cref{problem:inequality-relaxation0}. We will
prove $J_p  = {\bf Tr}(\Lambda S_p)\le {\bf Tr}(\Lambda S_{p,1} )
= J_{p,1}$. If we define the operator ${\cal H}(S):=A_{F_{p,1}}
SA_{F_{p,1}}^T  + \begin{bmatrix}
   I_n\\
   F_{p,1}\\
\end{bmatrix} Z \begin{bmatrix}
   I_n\\
   F_{p,1}\\
\end{bmatrix}^T$, then we can easily prove that ${\cal H}$ is ${\mathbb S}_+^{n+m}$-monotone, i.e., $P
\succeq P'\Rightarrow {\cal H}(P)\succeq {\cal H}(P')$. Repeatedly
applying the operator to both sides of ${\cal H}(S_{p,1})\preceq
S_{p,1}$ leads to ${\cal H}(S_{p,1})^k \preceq \cdots \preceq
{\cal H}(S_{p,1})^2 \preceq {\cal H}(S_{p,1})\preceq S_{p,1}$.
Since the sequence $\{ {\cal H}(S_{p,1})^k \}_{k =0}^\infty$ is
monotone and bounded, it converges, i.e., $\lim_{k \to \infty }
{\cal H}(S_{p,1})^k =:\bar S={\cal H}(\bar S)$. Then, $\bar S
\preceq S_{p,1}$, and hence, ${\bf Tr}(\Lambda \bar S)\leq {\bf
Tr}(\Lambda S_{p,1})=J_{p,1}$. Now, note that $(\bar S, F_{p,1})$
is a feasible solution to~\cref{problem:inequality-relaxation0},
we have $J_{p} = {\bf Tr}(\Lambda S_p) \leq {\bf Tr}(\Lambda \bar
S)$. Therefore, $J_p ={\bf Tr}(\Lambda S_p) \leq {\bf Tr}(\Lambda
\bar S)\leq {\bf Tr}(\Lambda S_{p,1})=J_{p,1}$, and the proof is
completed.
\end{proof}
Note that even though the optimal values
of~\cref{problem:inequality-relaxation0}
and~\cref{problem:inequality-relaxation} are identical, their
solutions may be different. Let $(S_p,F_p)$ and
$(S_{p,1},F_{p,1})$ be optimal solutions to~[Problem~3, main
document] and~\cref{problem:inequality-relaxation}, respectively.
By~\cref{proposition:equivalence}, ${\bf Tr}(\Lambda S_p)={\bf
Tr}(\Lambda S_{p,1})$. By~[Assumption~1, main document], $Q
\succeq 0$ implies that there may exist $S_{p,1} \succeq 0$ with
$A_F S_{p,1} A_F^T +
\begin{bmatrix}
 I_n \\ F \\
\end{bmatrix} Z \begin{bmatrix}
 I_n \\
 F \\
\end{bmatrix}^T \neq S_{p,1}$.

To proceed, we study a modification
of~\cref{problem:inequality-relaxation}, where $S \succeq 0$ is
replaced with the strict inequality $S \succ 0$.
\begin{problem}\label{problem:inequality-relaxation2}
Solve
\begin{align*}
&J_{p,2}:=\mathop {\inf}\limits_{S \in {\mathbb S}^{n + m} ,\,F
\in {\mathbb R}^{m \times n} } \,\,{\bf Tr}(\Lambda S)\\
&{\rm subject\,\,to}\quad S \succ 0,\\
&A_F SA_F^T + \begin{bmatrix}
 I_n \\
 F \\
\end{bmatrix} Z \begin{bmatrix}
 I_n \\
 F \\
\end{bmatrix}^T  \preceq S.
\end{align*}
\end{problem}
Since the semidefinite cone $\{S \in {\mathbb S}^{n+m}: S \succeq
0\}$ is the closure of $\{S \in {\mathbb S}^{n+m}: S \succ 0\}$,
the optimal value of~\cref{problem:inequality-relaxation2}
and~\cref{problem:inequality-relaxation} are identical. It is
stated in the following proposition.
\begin{proposition}\label{proposition:equivalence2}
The optimal value of~\cref{problem:inequality-relaxation} is
equivalent to that of~\cref{problem:inequality-relaxation2}, i.e.,
$J_{p}=J_{p,1}=J_{p,2}$.
\end{proposition}
\cref{proposition:equivalence2} is now in the form for
which~\cref{lemma:extended-Schur-lemma} can be applied. In the
following proposition, we propose an SDP-based optimal control
synthesis condition.
\begin{problem}\label{problem:SDP-design}
Solve
\begin{align*}
&(S^*,G^*,K^*):=\mathop{\arginf}_{S \in {\mathbb S}^{n+m},G \in
{\mathbb R}^{n \times n},K \in {\mathbb R}^{m\times n}} {\bf
Tr}(\Lambda S)\\
&{\rm subject\,\,to}\quad S \succ 0,\\
&\begin{bmatrix}
   S & \begin{bmatrix}
   G^T\\
   K \\
\end{bmatrix}\\
  \begin{bmatrix}
   G & K^T \\
\end{bmatrix} & G+G^T -\begin{bmatrix}
   A & B  \\
\end{bmatrix} S \begin{bmatrix}
   A & B  \\
\end{bmatrix}^T-Z \\
\end{bmatrix} \succeq 0
\end{align*}
\end{problem}
\begin{proposition}
If $(S^*,G^*,K^*)$ is an optimal solution
to~\cref{problem:SDP-design}, then $G^*$ is nonsingular, and $F^*
= K^*((G^*)^T )^{-1}$ is the optimal gain in~[(4), main document].
\end{proposition}
\begin{proof}
We apply~\cref{lemma:extended-Schur-lemma} to prove
that~\cref{problem:inequality-relaxation2} is equivalent to
\begin{align*}
&(S^*,F^*,G^*):=\mathop{\arginf}_{S \in {\mathbb S}^{n+m},G \in
{\mathbb R}^{n\times n},F \in {\mathbb R}^{m\times n}} {\bf
Tr}(\Lambda S)\\
&{\rm subject\,\,to}\quad S \succ 0,\\
&\begin{bmatrix}
   S & \begin{bmatrix}
   I  \\
   F  \\
\end{bmatrix} G^T \\
   G \begin{bmatrix}
   I & F^T \\
\end{bmatrix} & G+G^T- \begin{bmatrix}
   A & B  \\
\end{bmatrix} S \begin{bmatrix}
   A & B  \\
\end{bmatrix}^T-Z \\
\end{bmatrix} \succeq 0.
\end{align*}
For any feasible $(S,F)$, because $Z \succ 0$, we have
$\begin{bmatrix}
   A & B  \\
\end{bmatrix} S \begin{bmatrix}
   A & B  \\
\end{bmatrix}^T +Z \succ 0$, meaning that $G + G^T  \succ 0$
holds. This implies that $G$ is positive definite and nonsingular.
Letting $GF^T = K$, we can obtain~\cref{problem:SDP-design}. Since
$G^*$ is nonsingular, $F^*$ can be always recovered from the
solution $(S^*,G^*,K^*)$ by $F^* = K^*((G^*)^T )^{-1}$. From the
reasoning, one concludes that $(S^*,F^*)$ is also an optimal
solution to~\cref{problem:inequality-relaxation2}. Therefore, we
have $J_p =J_{p,1}=J_{p,2}={\bf Tr}(\Lambda S^*)$. Since
$(S^*,F^*)$ is feasible for~\cref{problem:inequality-relaxation2},
we also have
\begin{align*}
&A_{F^*}S^* A_{F^*}^T+ \begin{bmatrix}
   I_n \\
   F^*  \\
\end{bmatrix} Z \begin{bmatrix}
   I_n \\
   F^* \\
\end{bmatrix}^T \preceq  S^*
\end{align*}
If we define the operator
\begin{align*}
&{\cal H}(S):=A_{F^*}SA_{F^*}^T+ \begin{bmatrix}
   I_n \\
   F^* \\
\end{bmatrix} Z \begin{bmatrix}
   I_n \\
   F^* \\
\end{bmatrix}^T,
\end{align*}
then the last matrix inequality can be written by ${\cal
H}(S^*)\preceq S^*$. It can be easily proved that ${\cal H}$ is
${\mathbb S}_+^{n+m}$-monotone, i.e., $P \succeq P'\Rightarrow
{\cal H}(P)\succeq {\cal H}(P')$. Repeatedly applying the operator
to both sides of ${\cal H}(S^*)\preceq S^*$ leads to ${\cal
H}(S^*)^k \preceq \cdots \preceq {\cal H}(S^*)^2 \preceq {\cal
H}(S^*)\preceq S^*$. Since the sequence $\{ {\cal H}(S^*)^k \}_{k
=0}^\infty$ is monotone and bounded, it converges, i.e., $\lim_{k
\to \infty } {\cal H}(S^*)^k =:\bar S={\cal H}(\bar S)$. Then,
$\bar S \preceq S^*$, and hence, $J_p= J_{p,1}=J_{p,2}= {\bf
Tr}(\Lambda S^*)\ge {\bf Tr}(\Lambda \bar S)$. However, since
$(\bar S,F^*)$ is a feasible solution
to~\cref{problem:inequality-relaxation0}, we obtain $J_p=
J_{p,1}=J_{p,2}= {\bf Tr}(\Lambda S^*)\ge {\bf Tr}(\Lambda \bar S)
\ge J_p$. Therefore, $(\bar S,F^*)$ is an optimal solution
to~[Problem~3, main document]. By~[Proposition~1, main document],
$F^*$ is the optimal gain in[(4), main document]. This completes
the proof.
\end{proof}
From now on, we focus on optimal LQR control design with energy
and input constraints as stated in the following problem.
\begin{problem}\label{problem:energy-constrained-optimal-control}
Given constants $\rho>0,\gamma_i>0,i \in \{1,2,\ldots,n+m\}$,
solve
\begin{align}
&J_c:=\mathop{\inf}_{F \in {\cal F}}\sum_{i =
1}^r{J(F,z_i)}\nonumber\\
&{\rm subject\,\,to}\nonumber\\
&\sum_{i = 1}^r {\sum_{k = 0}^\infty  {x_j (k;F,z_i
)^2}}\le \gamma_j ,\quad j \in \{ 1,2, \ldots ,n\},\label{eq:energy-constraint1}\\
&\sum_{i=1}^r {\sum_{k=0}^\infty {u_{j-n}(k;F,z_i)^2}} \le
\gamma_j,\quad j \in \{ n + 1,n + 2,
\ldots ,n + m\},\label{eq:energy-constraint2}\\
&\| u(k;F,z) \|^2  \le \rho \| x(k;F,z) \|^2 ,\quad \forall k \in
N,\forall z \in {\mathbb R}^n,\label{eq:energy-constraint3}
\end{align}
where for any vector $x$, $x_j$ indicates the $j$-th element of
$x$ and $(u(k;F,z_i))_{k=0}^\infty$ denotes the input trajectory
with state feedback gain $F$ starting from the initial state
$z_i$.
\end{problem}
The constraint~\eqref{eq:energy-constraint1} is related to the
energy of the state trajectories,
and~\eqref{eq:energy-constraint2} corresponds to the energy of the
input trajectories. The constraint~\eqref{eq:energy-constraint3}
is a state-dependent input constraint. We note that similar energy
constraints were considered in~\cite{gattami2010generalized} as
well. However, the SDP condition in~\cite{gattami2010generalized}
cannot address the state-dependent input constraint. Similarly
to~\cref{problem:inequality-relaxation}, we first propose an
alternative form
of~\cref{problem:energy-constrained-optimal-control}.
\begin{problem}\label{problem:energy-constrained-optimal-control2}
Given constants $\rho>0,\gamma_i>0,i \in \{1,2,\ldots,n+m\}$,
solve
\begin{align*}
&J_{c,1}  = \mathop {\inf }\limits_{F \in F} {\bf{Tr}}(\Lambda
S)\\
&{\rm subject\,\,to}\\
&S\succeq 0,\\
&A_F SA_F^T +\begin{bmatrix}
   I_n\\
   F\\
\end{bmatrix} Z \begin{bmatrix}
   I_n\\
   F\\
\end{bmatrix}^T \preceq S\\
&e_i^T Se_i  \le \gamma _i ,\quad i \in \{ 1,2, \ldots ,n + m\},\\
&F^T F \preceq \rho I_n.
\end{align*}
\end{problem}
\begin{proposition}
The optimal value
of~\cref{problem:energy-constrained-optimal-control} is equivalent
to that of~\cref{problem:energy-constrained-optimal-control2},
i.e., $J_{c}=J_{c,1}$.
\end{proposition}
\begin{proof}
Following similar lines of the proof of~[Proposition~1, main
document],~\cref{problem:energy-constrained-optimal-control2} can
be equivalently expressed as
\begin{align}
&J_c=\mathop{\inf}_{F\in {\cal F}}{\bf Tr}(\Lambda S)\label{eq:alternative-form-of-problem}\\
&{\rm subject\,\,to}\nonumber\\
&S \succeq 0,\nonumber\\
&A_F SA_F^T  + \begin{bmatrix}
   I_n\\
   F\\
\end{bmatrix} Z \begin{bmatrix}
   I_n\\
   F \\
\end{bmatrix}^T=S,\nonumber\\
&e_i^T Se_i  \le \gamma _i ,\quad i \in \{ 1,2, \ldots ,n + m\},\nonumber\\
&F^T F \preceq \rho I_n,\nonumber
\end{align}
where $S = \sum_{i=1}^r {\sum_{k=0}^\infty  \begin{bmatrix}
   x(k;F,z_i)\\
   Fx(k;F,z_i)\\
\end{bmatrix} \begin{bmatrix}
   x(k;F,z_i)\\
   Fx(k;F,z_i)\\
\end{bmatrix}^T}$. Note that the constraints $\|u(k)\|^2 \le\rho \|x(k) \|^2,k\in {\mathbb
N}$, in~\cref{problem:energy-constrained-optimal-control} can be
expressed as $x(k;F,z)^T (F^T F - \rho I_n )x(k;F,z) \le 0,
\forall k \in {\mathbb N},\forall z \in {\mathbb R}^n$, and since
$x(k;F,z)\forall k \in {\mathbb N},\forall z \in {\mathbb R}^n$
spans ${\mathbb R}^n$, the last inequality is equivalent to the
linear matrix inequality (LMI) $F^T F \preceq \rho I_n $. Now, we
follows arguments similar to the proof
of~\cref{proposition:equivalence}. The feasible set
of~\cref{problem:energy-constrained-optimal-control2} is larger
than that of~\eqref{eq:alternative-form-of-problem}, we have
$J_{c,1} \leq J_c$. To prove the reversed inequality, one can
follow similar arguments of~\cref{proposition:equivalence}. This
completes the proof.
\end{proof}

We propose an SDP-based design algorithm to
solve~\cref{problem:energy-constrained-optimal-control}.
\begin{problem}\label{problem:energy-constrained-optimal-control-SDP}
Given constants $\rho>0,\gamma_i>0,i \in \{1,2,\ldots,n+m\}$,
solve
\begin{align*}
&(\bar S,\bar G,\bar K):=\mathop{\arginf}_{S \in {\mathbb
S}^{n+m},G \in {\mathbb R}^{n \times n},K \in {\mathbb R}^{m\times
n}} {\bf
Tr}(\Lambda S)\\
&{\rm subject\,\,to}\quad S \succ 0,\\
&\begin{bmatrix}
   S & \begin{bmatrix}
   G^T\\
   K \\
\end{bmatrix}\\
  \begin{bmatrix}
   G & K^T \\
\end{bmatrix} & G+G^T -\begin{bmatrix}
   A & B  \\
\end{bmatrix} S \begin{bmatrix}
   A & B  \\
\end{bmatrix}^T-Z \\
\end{bmatrix} \succeq 0,\\
&\begin{bmatrix}
   \rho I & K \\
   K^T & G+G^T-I \\
\end{bmatrix} \succeq 0,\\
&e_i^T Se_i\le\gamma_i,\quad i \in \{1,2,\ldots,n+m\}.
\end{align*}
\end{problem}
\begin{proposition}
If $(\bar S,\bar G,\bar K)$ is an optimal solution
to~\cref{problem:energy-constrained-optimal-control-SDP}, then
$\bar G$ is nonsingular, and $\bar F = \bar K((\bar G)^T )^{-1}$
is a suboptimal gain that
solves~\cref{problem:energy-constrained-optimal-control-SDP}. In
other words, the state-input trajectories with the feedback gain
$\bar F$ satisfy the
constraints~\eqref{eq:energy-constraint1},\eqref{eq:energy-constraint3}
and ${\bf Tr}(\Lambda \bar S)$ is an upper bound on $J_c$.
\end{proposition}
\begin{proof}
We first
convert~\cref{problem:energy-constrained-optimal-control2} into
\begin{align}
&J_{c,2}=\mathop{\inf}_{F \in {\cal F}} {\bf Tr}(\Lambda S)\label{eq:alternative-form-of-problem2}\\
&{\rm subject\,\,to}\nonumber\\
&S \succ 0,\nonumber\\
&A_F SA_F^T +\begin{bmatrix}
   I_n\\
   F\\
\end{bmatrix} Z \begin{bmatrix}
   I_n\\
   F\\
\end{bmatrix}^T \preceq S\nonumber\\
&e_i^T Se_i  \le \gamma _i ,\quad i \in \{ 1,2, \ldots ,n + m\},\nonumber\\
&F^T F \preceq \rho I_n.\nonumber
\end{align}
Since the semidefinite cone $\{S \in {\mathbb S}^{n+m}: S \succeq
0\}$ is the closure of $\{S \in {\mathbb S}^{n+m}: S \succ 0\}$,
the optimal value of~\eqref{eq:alternative-form-of-problem2}
and~\cref{problem:energy-constrained-optimal-control2} are
identical, i.e., $J_{c,2}=J_{c,1}=J_{c}$.
Applying~\cref{lemma:extended-Schur-lemma} yields the equivalent
condition
\begin{align*}
&(S,G_1,G_2,F): = \mathop {\arginf }_{S \in {\mathbb S}^{n + m}
,G_1\in {\mathbb R}^{n \times n},G_2\in
{\mathbb R}^{n \times n} ,F \in {\mathbb R}^{m \times n}} {\bf Tr}(\Lambda S)\\
&{\rm{subject}}\,\,{\rm{to}}\quad S \succ 0,\\
&\begin{bmatrix}
   S & \begin{bmatrix}
   G_1^T \\
   FG_1^T \\
\end{bmatrix} \\
   \begin{bmatrix}
   G_1 & G_1 F^T \\
\end{bmatrix} & G_1 + G_1^T  - \begin{bmatrix}
   A & B  \\
\end{bmatrix} S \begin{bmatrix}
   A & B  \\
\end{bmatrix}^T-Z \\
\end{bmatrix} \succeq 0,\\
&\begin{bmatrix}
   \rho I_m & FG_2^T   \\
   G_2 F^T  & G_2+ G_2^T - I_n\\
\end{bmatrix} \succeq 0,\\
&e_i^T Se_i  \le \gamma _i ,\quad i \in \{ 1,2, \ldots ,n + m\},
\end{align*}
Letting $G_1 = G_2$ and introducing the change of variables $FG^T
= K$, we obtain the SDP
in~\cref{problem:energy-constrained-optimal-control-SDP}. Finally,
we note that since we introduced an additional constraint $G_1 =
G_2$, the optimal value
of~\cref{problem:energy-constrained-optimal-control-SDP} is an
upper bound on $J_c$
in~\cref{problem:energy-constrained-optimal-control}
and~\cref{problem:energy-constrained-optimal-control2}. Finally,
using similar lines as in the proof
of~\cref{problem:inequality-relaxation2}, we can easily prove that
the state-input trajectories satisfy the
constraints~\eqref{eq:energy-constraint1}
and~\eqref{eq:energy-constraint3}. This completes the proof.
\end{proof}
\begin{example}
Consider the system $A = \begin{bmatrix}
   1 & 1  \\
   0 & 1  \\
\end{bmatrix},B = \begin{bmatrix}
   0  \\
   1  \\
\end{bmatrix}$ with $Q = I_2$, $R = 0.1$, $Z = I_2$.
Solving~\cref{problem:SDP-design} results in the gain $F^*=
\begin{bmatrix}
-0.5792 & -1.5456\\
\end{bmatrix}$ and the cost ${\bf Tr}(\Lambda S^*)=
5.5499$. Both are identical to the results from the standard LQR
design approach using ARE. In addition, we
solved~\cref{problem:energy-constrained-optimal-control-SDP} with
$\gamma_1=\cdots =\gamma_{n+m}=5$ and different $\rho  \in
[1.2,5]$, and the resulting optimal costs ${\bf Tr}(\Lambda \bar
S)$ are depicted in~\cref{fig:ex-fig5}. At some point around $\rho
= 1.2$,~\cref{problem:energy-constrained-optimal-control-SDP}
becomes infeasible.
\begin{figure}[t]
\centering\includegraphics[width=7cm,height=4cm]{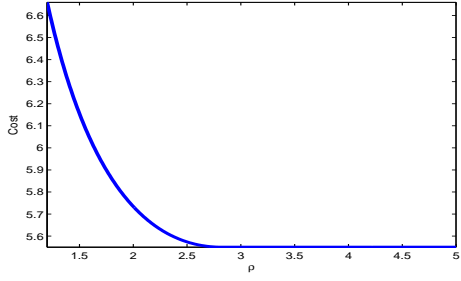}
\caption{The optimal cost obtained
using~\cref{problem:energy-constrained-optimal-control-SDP} with
$\gamma_1=\cdots =\gamma_{n+m}=5$ and different $\rho  \in
[1.2,5]$.}\label{fig:ex-fig5}
\end{figure}

\end{example}

\appendices

\section{Explicit dual problem}
In this section, we derive an explicit optimization form of the
dual problem~[(12), main document], which is a convex semidefinite
programming problem (SDP).
\begin{proposition}
Consider the problem
\begin{align}
&\mathop{\sup}\limits_{P \in {\mathbb S}_+^{n + m},\,P_{22} \succ
0} {\bf Tr}(Z(P_{11} - P_{12} P_{22}^{-1}P_{12}^T))\,\,{\rm subject\,\,to}\nonumber\\
&\begin{bmatrix}
\begin{bmatrix}
   A & B  \\
\end{bmatrix}^T P_{11} \begin{bmatrix}
   A & B\\
\end{bmatrix} - P + \Lambda & \begin{bmatrix}
   A & B  \\
\end{bmatrix}^T P_{12}\\
P_{12}^T \begin{bmatrix}
   A & B \\
\end{bmatrix} & P_{22}\\
\end{bmatrix} \succeq 0.\label{eq:SDP-constraint}
\end{align}
It is a convex optimization problem and is an explicit form of the
dual problem~[(12), main document].
\end{proposition}
\begin{proof}
Consider the dual function in~[(14), main document], and notice
that the dual optimal solution $P^*$ of~[(15), main document]
satisfies $P_{22}^*\succ 0$ and $-(P_{22}^*)^{-1}(P_{12}^*)^T \in
{\cal F}$. Therefore, one can restrict the set $\cal P$ in~[(14),
main document] to the subset ${\cal P}':=\{P\in {\mathbb
S}_+^{n+m}:A_F^T PA_F-P+\Lambda \succeq 0,\forall F\in {\cal
F},P_{22} \succ 0,-P_{22}^{-1} P_{12}^T\in {\cal F}\}$ without
changing the optimal dual function value. In this case, the
infimum in~[(14), main document] is attained at $F = -P_{22}^{-1}
P_{12}^T$. By plugging $-P_{22}^{-1}P_{12}^T$ into $F$ in the dual
problem~[(15), main document], it is equivalently converted to
\begin{align*}
&\mathop{\sup}\limits_{P \in {\mathbb S}_+^{n+m},\,P_{22}\succ 0}
{\bf Tr}(Z(P_{11}-P_{12}P_{22}^{-1}P_{12}^T))\,\,{\rm subject\,\,to}\\
&\begin{bmatrix}
   A & B \\
\end{bmatrix}^T (P_{11}-P_{12} P_{22}^{-1} P_{12}^T)\begin{bmatrix}
   A & B \\
\end{bmatrix}-P+\Lambda \succeq 0.
\end{align*}
The linear matrix inequality~\eqref{eq:SDP-constraint} can be
obtained by using the Schur complement in~[Lemma~3, main
document]. The problem is a convex optimization problem because
the objective function can be replaced by $t \in {\mathbb R}$ with
additional constraints $\sum\nolimits_{i = 1}^r {z_i^T (P_{11}  -
P_{12} P_{22}^{ - 1} P_{12}^T )z_i }  \ge t$, and the last
inequality can be converted to a linear matrix inequality using
the Schur complement $r$ times. In particular, applying the Schur
complement leads to
\begin{align*}
&\begin{bmatrix}
   { - t + \sum\limits_{i = 1}^r {z_i^T P_{11} z_i }  + \sum\limits_{i = 2}^r {z_i^T P_{12} P_{22}^{ - 1} P_{12}^T z_i } } & {z_1^T P_{12} }  \\
   {P_{12}^T z_1 } & {P_{22} }  \\
\end{bmatrix} \succeq 0.
\end{align*}
The left-hand side can be decomposed into
\begin{align*}
&\begin{bmatrix}
   { - t + \sum\limits_{i = 1}^r {z_i^T P_{11} z_i }  + \sum\limits_{i = 3}^r {z_i^T P_{12} P_{22}^{ - 1} P_{12}^T z_i } } & {z_1^T P_{12} }  \\
   {P_{12}^T z_1 } & {P_{22} }  \\
\end{bmatrix}\\
& + \begin{bmatrix}
   z_2^T P_{12}  \\
   0  \\
\end{bmatrix} P_{22}^{-1} \begin{bmatrix}
   z_2^T P_{12} \\
   0  \\
\end{bmatrix}^T,
\end{align*}
and the Schur complement can be applied again. Repeating this
$r-1$ times, one gets a linear matrix inequality constraint. This
completes the proof.
\end{proof}

\bibliographystyle{IEEEtran}
\bibliography{reference}

\end{document}